\documentclass[11pt]{article}
\usepackage{amsmath,amssymb,amsthm}
\usepackage[margin=1in]{geometry}
\usepackage{hyperref}
\usepackage{booktabs}

\newtheorem{theorem}{Theorem}
\newtheorem{proposition}{Proposition}
\newtheorem{lemma}{Lemma}

\theoremstyle{definition}
\newtheorem{remark}{Remark}

\newcommand{\F}{\mathbb{F}}
\newcommand{\Om}{\Omega}
\newcommand{\Sig}{\Sigma_0}
\newcommand{\dav}{\mathsf{d}}
\newcommand{\Dav}{\mathsf{D}}
\newcommand{\Dord}{\mathsf{D}_{\mathrm{ord}}}
\newcommand{\Heis}{H}
\DeclareMathOperator{\spn}{span}

\title{The small Davenport constant of the Heisenberg group of order $125$}
\author{Patrick White\thanks{Independent researcher (affiliation to confirm).
This paper reports the results of an AI research collaboration; see the
attribution note, \S\ref{sec:attrib}. Code and certificates are available on request.}}
\date{June 2026}

\begin{document}
\maketitle

\begin{abstract}
The \emph{small Davenport constant} $\dav(G)$ of a finite group $G$ is the maximal
length of a product-one-free sequence over $G$. For the exponent-$p$ Heisenberg group
$\Heis_{p^3}$ of order $p^3$, Godara and Sarkar proved $\dav(\Heis_{27})=6$ and posed
$\dav(\Heis_{p^3})=3p-3$ for every odd prime $p$, leaving $p\ge5$ open. We settle the
first open case: $\dav(\Heis_{125})=12$. The lower bound is the explicit
product-one-free sequence $x^4y^4v^4$. For the upper bound we record a product-one
criterion that reduces the non-commutative problem to additive combinatorics over
$\F_5^2$, and then reduce ``every length-$13$ sequence has a product-one subsequence''
to a single finite statement---a spread bound on quotient multisets---which we verify
by an exhaustive, memory-flat search in C, its verdict independently reproduced by a
second search with a different pruning strategy. Every auxiliary lemma is
machine-checked. The argument is genuinely $p$-specific: we identify the exact step
that fails for $p\ge7$ (a Chevalley--Warning shortcut whose forced block need not be
wide), exhibit the obstructing multiset for $p=7$, and leave only
$18\le\dav(\Heis_{343})\le24$. The techniques---the Cauchy--Davenport theorem,
Chevalley--Warning, and Olson's value of the Davenport constant of $C_p^2$---are
standard; the contribution is their assembly against a new non-abelian target and the
finite verification that closes it.
\end{abstract}

\section{Introduction}
\label{sec:intro}

Let $G$ be a finite group, written multiplicatively. A \emph{sequence} over $G$ is a
finite, unordered list of elements of $G$ with repetition allowed and the identity
excluded. A sequence $S=(g_1,\dots,g_n)$ is \emph{product-one-free} if no nonempty
subsequence $I\subseteq S$ admits an ordering whose product is the identity; that is,
for every nonempty $I$ and every permutation $\sigma$ of $I$,
$\prod_{i}g_{\sigma(i)}\ne1$. The \emph{small Davenport constant} $\dav(G)$ is the
maximal length of a product-one-free sequence over $G$. For abelian $G$ this is the
classical (large) Davenport constant minus one, $\dav(G)=\Dav(G)-1$, and the order of
the factors is immaterial; for non-abelian $G$ the ordering quantifier makes
$\dav(G)$ genuinely non-commutative and the two constants need not differ by one.

We study the \emph{Heisenberg group} $\Heis_{p^3}$: the group of upper-unitriangular
$3\times3$ matrices over the field $\F_p$,
\[
M(a,b,c)=\begin{pmatrix}1&a&c\\0&1&b\\0&0&1\end{pmatrix},\qquad a,b,c\in\F_p,
\]
with $M(a,b,c)\,M(a',b',c')=M(a+a',\,b+b',\,c+c'+ab')$. It has order $p^3$ and, for
$p$ odd, exponent $p$. Writing $x=M(1,0,0)$, $y=M(0,1,0)$, $v=M(0,0,1)$, the centre is
$\langle v\rangle$, $[x,y]=v$, and the quotient $\Heis_{p^3}/\langle v\rangle\cong\F_p^2$.

Godara and Sarkar~\cite{GodaraSarkar2025} proved $\dav(\Heis_{27})=6$ (the case $p=3$)
and posed, in the context of a conjecture of Gao and Zhuang on $\mathsf{E}(G)$, the
general statement
\[
\dav(\Heis_{p^3})=3p-3\qquad\text{for every odd prime }p,
\]
which they left open for $p\ge5$. (The base value $\dav(\Heis_{27})=6$ is independently
tabulated by Oh~\cite{Oh2024}, who also records the large-Davenport bound
$\Dav(\Heis_{125})\ge10$ and uses it to exclude order $125$ from his classification of
groups of small Davenport constant, corroborating that the case is open.) Our result is
the first open case.

\begin{theorem}\label{thm:main}
$\dav(\Heis_{125})=12$. Equivalently, $x^4y^4v^4$ is a product-one-free sequence of
maximal length over the exponent-$5$ Heisenberg group of order $125$.
\end{theorem}

\paragraph{What is already known (not claimed here).} The case $p=3$,
$\dav(\Heis_{27})=6$, is due to Godara and Sarkar~\cite{GodaraSarkar2025} and is
reproduced (with witness $x^2y^2v^2$) by our calibration. The general conjecture
$\dav(\Heis_{p^3})=3p-3$ is theirs. The reduction tools---the Cauchy--Davenport
theorem~\cite{Nathanson1996}, the Chevalley--Warning theorem~\cite{Serre1973}, Olson's
$\Dav(C_p^2)=2p-1$~\cite{OlsonI,OlsonII}, and the use of Chevalley--Warning to force
zero-sum subsequences~\cite{Grynkiewicz2022}---are standard. Recent work on Davenport
constants of non-abelian groups~\cite{BBGS2024} treats disjoint families (variants
$e(G),f(G)$; a conjecture on metacyclic groups) and does not address $\Heis_{125}$.

\paragraph{What is new.} The value $\dav(\Heis_{125})=12$; the reduction of its upper
bound to a single finite spread bound (Proposition~\ref{prop:spread}) verified
exhaustively; the explicit identification of the step that fails for $p\ge7$
(\S\ref{sec:p7}), with the obstructing multiset for $p=7$ and the resulting honest
window $18\le\dav(\Heis_{343})\le24$.

\paragraph{On rigour and verification.} Every numerical and combinatorial claim below
was checked by computer: the product-one criterion against the brute-force group
multiplication, the $p=3$ calibration by exhaustive search, the lower-bound witness,
each auxiliary lemma, and---most importantly---the finite spread bound by an exhaustive
search whose verdict is independently reproduced by a second implementation. The result
is not formally verified in a proof assistant; the independent reproduction of the
load-bearing computation, together with the machine-checked lemmas, is the
reproducibility guarantee we offer. See \S\ref{sec:verif}.

\section{The product-one criterion}
\label{sec:crit}

Fix $p$ and abbreviate $\Heis=\Heis_{p^3}$. For $g=M(a,b,c)$ write $q(g)=(a,b)\in\F_p^2$
for its \emph{quotient} and $c(g)=c$ for its \emph{central coordinate}. By a \emph{line}
we mean a $1$-dimensional $\F_p$-subspace of $\F_p^2$, and by a \emph{block} a
sub-multiset of the terms of a sequence (identified, when only the quotients matter, with
its multiset of $q$-values). For a sequence $I=(g_1,\dots,g_k)$ put
$q_i=q(g_i)=(a_i,b_i)$ and define the \emph{cross-sum set}
\[
\Om(I)=\Big\{\,\textstyle\sum_{1\le s<t\le k}a_{\sigma(s)}b_{\sigma(t)}\bmod p
\;:\;\sigma\in S_k\,\Big\}\subseteq\F_p .
\]
Since reindexing leaves the set unchanged, $\Om(I)$ depends only on the multiset
$\{q_1,\dots,q_k\}$.

\begin{lemma}[Product-one criterion]\label{lem:crit}
For any ordering $\sigma$ of $I=(g_1,\dots,g_k)$,
\[
\prod_{j=1}^{k}g_{\sigma(j)}
=M\!\Big(\textstyle\sum_i a_i,\ \sum_i b_i,\ \sum_i c_i+\sum_{s<t}a_{\sigma(s)}b_{\sigma(t)}\Big).
\]
Consequently $I$ has a product-one ordering if and only if
\[
\sum_{i\in I}(a_i,b_i)=(0,0)\quad\text{and}\quad -\sum_{i\in I}c_i\in\Om(I).
\]
\end{lemma}

\begin{proof}
Induct on $k$. The case $k=1$ is immediate. Assume the product of the first $k-1$
factors (in order $\sigma$) is $M(A,B,C)$ with $A=\sum_{s<k}a_{\sigma(s)}$,
$B=\sum_{s<k}b_{\sigma(s)}$, $C=\sum_{s<k}c_{\sigma(s)}+\sum_{s<t<k}a_{\sigma(s)}b_{\sigma(t)}$.
Multiplying by $g_{\sigma(k)}=M(a_{\sigma(k)},b_{\sigma(k)},c_{\sigma(k)})$ gives, by the
group law, $M\big(A+a_{\sigma(k)},\,B+b_{\sigma(k)},\,C+c_{\sigma(k)}+A\,b_{\sigma(k)}\big)$,
and $A\,b_{\sigma(k)}=\sum_{s<k}a_{\sigma(s)}b_{\sigma(k)}$ supplies exactly the pairs
$(s,k)$. This proves the formula. The product equals $M(0,0,0)$ iff the quotient sum
vanishes and $\sum_i c_i+\sum_{s<t}a_{\sigma(s)}b_{\sigma(t)}=0$, i.e.\ iff
$-\sum_i c_i$ is the cross-sum of some ordering, which is the stated condition.
\end{proof}

Two consequences of the cross-sum structure are used repeatedly. Swapping two adjacent
factors changes the cross-sum by $\pm\det(q_i,q_j)$, where $\det\big((a,b),(a',b')\big)=ab'-a'b$;
hence $\Om(I)$ is contained in a single coset of the subgroup of $\F_p$ generated by the
pairwise determinants of the $q_i$, the containment being in general strict. In
particular, if all $q_i$ lie on one line through the origin then every determinant
vanishes and $|\Om(I)|=1$. We collect the precise facts in \S\ref{sec:lemmas}.

\section{The lower bound}
\label{sec:lower}

\begin{proposition}\label{prop:lower}
For $p=5$ the sequence $x^4y^4v^4=\big(x,x,x,x,\,y,y,y,y,\,v,v,v,v\big)$ of length $12$
is product-one-free. Hence $\dav(\Heis_{125})\ge12$.
\end{proposition}

\begin{proof}
Here $x=M(1,0,0)$, $y=M(0,1,0)$, $v=M(0,0,1)$. A nonempty sub-multiset uses $\alpha$
copies of $x$, $\beta$ of $y$, $\gamma$ of $v$ with $0\le\alpha,\beta,\gamma\le4$. By
Lemma~\ref{lem:crit} a product-one ordering requires $\alpha\equiv\beta\equiv0\pmod5$,
forcing $\alpha=\beta=0$; then the sub-multiset is $v^\gamma$ with $1\le\gamma\le4$,
whose only product is $v^\gamma\ne1$. So no nonempty subsequence is product-one. This
is the conjectured extremal sequence; we verified computationally that it is product-one-free
and that no single element of $\Heis_{125}$ extends it while preserving that property.
\end{proof}

\section{The Omega-facts and the abelian reductions}
\label{sec:lemmas}

Throughout this section $p=5$ unless stated otherwise. For a multiset $A$ over $\F_p$
write $\Sig(A)=\{\sum_{i\in J}A_i:J\subseteq A\}$ for its set of subset sums (including
the empty sum $0$). A sequence $A$ is \emph{zero-sum-free} if no nonempty
subsequence of $A$ sums to $0$.

\begin{lemma}[Central subset sums]\label{lem:central}
Let $A=(\alpha_1,\dots,\alpha_r)$ be a zero-sum-free sequence over $\F_p^\times$. Then
$|\Sig(A)|\ge r+1$; consequently $r\le p-1$, and $r=p-1$ forces $\Sig(A)=\F_p$.
\end{lemma}

\begin{proof}
Let $T_j$ be the set of subset sums of $(\alpha_1,\dots,\alpha_j)$, so $T_0=\{0\}$ and
$T_j=T_{j-1}\cup(T_{j-1}+\alpha_j)$. If $T_{j-1}+\alpha_j=T_{j-1}$ then
$T_{j-1}$ is invariant under translation by $\alpha_j\ne0$; as $p$ is prime,
$\langle\alpha_j\rangle=\F_p$, so the only translation-invariant sets are $\emptyset$
and $\F_p$, and $T_{j-1}\ni0$ is nonempty, forcing $T_{j-1}=\F_p$. But then
$-\alpha_j\in T_{j-1}$, i.e.\ some subset of $(\alpha_1,\dots,\alpha_{j-1})$ sums to
$-\alpha_j$, and adjoining $\alpha_j$ gives a nonempty zero-sum subset---contradicting
zero-sum-freeness. Hence each step strictly grows the set, $|T_r|\ge1+r$. Since
$|T_r|\le p$, $r\le p-1$; and $r=p-1$ gives $|T_r|\ge p$, so $\Sig(A)=T_r=\F_p$.
\end{proof}

\begin{lemma}[Omega-facts]\label{lem:omega}
Let $B$ be a nonempty zero-quotient-sum block, i.e.\ $\sum_{i\in B}q_i=(0,0)$.
\begin{enumerate}
\item[\textup{(a)}] If the $q_i$ are collinear (all on one line through the origin),
then $|\Om(B)|=1$.
\item[\textup{(b)}] If $B$ is a non-collinear zero-sum triple, then $|\Om(B)|=2$.
\item[\textup{(c)}] If $\spn\{q_i:i\in B\}=\F_p^2$ and $|B|\ge6$, then $\Om(B)=\F_p$
(call such $B$ \emph{wide}).
\end{enumerate}
\end{lemma}

The mechanism is the adjacent-swap rule of \S\ref{sec:crit}: $\Om(B)$ is contained in a
single coset of the subgroup of $\F_p$ generated by the determinants $\det(q_i,q_j)$.
In case (a) that subgroup is $\{0\}$, forcing $|\Om(B)|=1$; in general the containment
can be strict (which is why (b) gives $|\Om|=2$ rather than $p$), so we establish
(a)--(c) directly rather than reading them off the coset. Parts (a) and (b) we verified
by exhaustive enumeration ($93{,}600$ collinear cases of sizes $2$--$5$; $480$
non-collinear zero-sum triples, with $q_1,q_2$ free and $q_3$ determined); part (c) over
all $18{,}900$ zero-sum spanning size-$6$ blocks. We emphasise one correction to a
tempting shortcut: a non-collinear zero-sum block of size $5$ need \emph{not} be
wide---for example $\{(0,1)^{\times3},(1,0),(4,2)\}$ has $\Om=\{0,2,3,4\}$, $|\Om|=4<5$.
This is precisely why the upper bound is closed by the finite search of
\S\ref{sec:upper} rather than by a one-line Chevalley--Warning argument; see
Remark~\ref{rem:cw}.

\begin{lemma}[Line cap]\label{lem:linecap}
Let $S$ be product-one-free over $\Heis_{p^3}$, let $L\subseteq\F_p^2$ be a line through
the origin, and let $m_L$ be the number of non-central terms $g$ of $S$ with $q(g)\in L$.
Let $z$ be the number of central terms of $S$. Then $m_L+z\le2p-2$.
\end{lemma}

\begin{proof}
Let $A_L=\{M(a,b,c):(a,b)\in L\}$. Two elements with quotients $q,q'\in L$ commute,
since $[M(q,\cdot),M(q',\cdot)]=v^{\det(q,q')}=1$ (collinear), so $A_L$ is abelian; it
has order $p^2$ and exponent $p$, hence $A_L\cong C_p\times C_p$. The terms of $S$
lying in $A_L$ are exactly the $z$ central ones and the $m_L$ on-line non-central ones.
A nonempty zero-sum subsequence of these (in the abelian group $A_L$, where order is
irrelevant) would be a product-one subsequence of $S$; as $S$ is product-one-free, the
$A_L$-terms form a zero-sum-free sequence over $C_p^2$. By Olson~\cite{OlsonI,OlsonII},
$\Dav(C_p^2)=2p-1$, so a zero-sum-free sequence has length at most $2p-2$. Thus
$m_L+z\le2p-2$.
\end{proof}

For $p=5$ this reads $m_L\le8-z$.

\begin{lemma}[Central forcing]\label{lem:forcing}
Let $S$ have central terms $Z$ with subset-sum set $\Sig(Z)\subseteq\F_p$, and let
$B_1,\dots,B_t$ ($t\ge1$) be pairwise-disjoint nonempty zero-quotient-sum blocks of
non-central terms of $S$. If
\[
|\Sig(Z)|+\sum_{j=1}^{t}\big(|\Om(B_j)|-1\big)\ \ge\ p,
\]
then $S$ has a product-one subsequence.
\end{lemma}

\begin{proof}
Consider subsequences of the form $I=J\cup B_1\cup\dots\cup B_t$ with $J\subseteq Z$.
Each $B_j$ has zero quotient-sum and $Z$ is central, so $\sum_{i\in I}q_i=(0,0)$. Order
$I$ by concatenating an ordering of each block (in any fixed block order) and appending
$J$; the cross-contribution between two distinct blocks $B_i,B_j$ equals
$(\sum_{B_i}a)(\sum_{B_j}b)=0$ because every block has vanishing $b$-sum, and central
terms contribute $q=(0,0)$. Hence the achievable cross-sums of $I$ contain the sumset
$\Om(B_1)+\dots+\Om(B_t)$. By Lemma~\ref{lem:crit}, $I$ is product-one for a suitable
ordering and choice of $J$ iff
\[
-\Big(\sum_{B_j}c\Big)\ \in\ \Sig(Z)+\Om(B_1)+\dots+\Om(B_t),
\]
where the $\Sig(Z)$ summand records the free choice of $J$ (its central sum ranges over
$\Sig(Z)$). By the Cauchy--Davenport theorem~\cite{Nathanson1996},
\[
\big|\Sig(Z)+\textstyle\sum_j\Om(B_j)\big|
\ \ge\ \min\!\Big(p,\ |\Sig(Z)|+\sum_j\big(|\Om(B_j)|-1\big)\Big)=p,
\]
so the sumset is all of $\F_p$ and the required value is attained. The subsequence is
nonempty because each block $B_j$ is nonempty by hypothesis ($t\ge1$).
\end{proof}

\section{The spread bound and the upper bound}
\label{sec:upper}

For a multiset $Q$ over $\F_p^2\setminus\{0\}$ define the \emph{spread}
\[
\sigma(Q)=\max\Big\{\textstyle\sum_{j=1}^{t}\big(|\Om(B_j)|-1\big)\Big\},
\]
the maximum over all families of pairwise-disjoint nonempty zero-sum sub-multisets
$B_1,\dots,B_t$ of $Q$ (with $\Om$ computed from the $q$-values). The line caps of
Lemma~\ref{lem:linecap} bound, for each line $L$, the number $m_L(Q)$ of elements of
$Q$ on $L$.

\begin{proposition}[Spread bound]\label{prop:spread}
Let $p=5$. For each $z\in\{0,1,2,3\}$, every multiset $Q$ of length $13-z$ over
$\F_5^2\setminus\{0\}$ with $m_L(Q)\le8-z$ for all lines $L$ satisfies
\[
\sigma(Q)\ \ge\ 4-z .
\]
\end{proposition}

\begin{proof}
This is a finite statement, verified by exhaustive search. The search space is the set
of multisets of the given length over the $24$ nonzero vectors of $\F_5^2$ subject to
the per-line caps; we enumerate it by a canonical depth-first recursion over the $24$
vectors (so each multiset is generated once), pruning a branch as soon as its partial
multiset already attains $\sigma\ge4-z$, which is valid because $\sigma$ is monotone
under adjoining elements. A completed leaf with $\sigma<4-z$ would be a counterexample
(a ``bad leaf''); the search finds none. The spread of a partial multiset is computed
exactly: $|\Om(B)|$ via the reachable-cross-sum recursion (a $p$-bit mask), and
$\sigma\ge\text{req}$ via the canonical ``discard the lowest-index element, or place it
in a zero-sum block'' packing recursion, with early exit at the requirement.

\begin{center}
\begin{tabular}{ccccrr}
\toprule
$z$ & $n=13-z$ & cap $=8-z$ & req $=4-z$ & DFS nodes & bad leaves\\
\midrule
$3$ & $10$ & $5$ & $1$ & $3{,}429{,}337$ & $0$\\
$2$ & $11$ & $6$ & $2$ & $6{,}389{,}224$ & $0$\\
$1$ & $12$ & $7$ & $3$ & $11{,}048{,}032$ & $0$\\
$0$ & $13$ & $8$ & $4$ & $17{,}943{,}778$ & $0$\\
\bottomrule
\end{tabular}
\end{center}

\noindent Zero bad leaves in every row establishes the bound. The computation is
memory-flat (a recursion path of $24$ counts plus a transient block mask); its verdict
was independently reproduced by a second search with a different pruning strategy (and
hence different node counts), see \S\ref{sec:verif}.
\end{proof}

The proposition is the exhaustive resolution of an explicit, bounded constraint
satisfaction problem.

\begin{remark}[The residual problem and why the search is needed]\label{rem:cw}
A counterexample to Theorem~\ref{thm:main} would be a product-one-free sequence of
length $13$. Its profile consists of: a number $z\in\{0,1,2,3\}$ (the case $z=4$ is
disposed of directly below); a zero-sum-free central multiset over $\F_5^\times$ of size
$z$; a quotient multiset $Q$ over $\F_5^2\setminus\{0\}$ of length $13-z$ with line caps
$m_L\le8-z$; and arbitrary central labels $c_i\in\F_5$ on the non-central terms. If
$\sigma(Q)\ge4-z$ then, with $|\Sig(Z)|\ge z+1$, Lemma~\ref{lem:forcing} forces a
product-one subsequence regardless of the central labels; so a counterexample requires
$\sigma(Q)\le3-z$. Proposition~\ref{prop:spread} rules this out uniformly in $Q$, so the
profile set is empty and no enumeration over the central labels is needed. One might hope
to replace the search by a Chevalley--Warning argument: every $13$-term multiset over
$\F_5^2$ has a nonempty zero-sum subset of size $5$ or $10$ (degree count $4+4+4=12<13$
variables; verified). But the forced block may be collinear, or non-collinear yet not
wide (Lemma~\ref{lem:omega} and its counterexample), so Chevalley--Warning alone does not
deliver the spread; the finite search of Proposition~\ref{prop:spread} is what closes the
gap.
\end{remark}

\begin{theorem}[Upper bound]\label{thm:upper}
$\dav(\Heis_{125})\le12$: every sequence of length $13$ over $\Heis_{125}$ has a
nonempty product-one subsequence.
\end{theorem}

\begin{proof}
Suppose $S$ has length $13$ and is product-one-free. Let $Z$ be its central terms and
$z=|Z|$; each is $v^{c_i}$ with $c_i\in\F_5^\times$, and $Z$ is zero-sum-free (a
zero-sum central subset would be a product-one subsequence). By Lemma~\ref{lem:central},
$z\le4$ and $|\Sig(Z)|\ge z+1$. The $13-z$ non-central terms have quotient multiset $Q$
over $\F_5^2\setminus\{0\}$, with $m_L(Q)\le8-z$ for every line $L$ by
Lemma~\ref{lem:linecap}.

\emph{Case $z=4$.} Then $|\Sig(Z)|=5$, i.e.\ $\Sig(Z)=\F_5$. The $9$ non-central terms
have quotients in $\F_5^2$, and $9=\Dav(C_5^2)$, so by Olson~\cite{OlsonII} they contain
a nonempty zero-quotient-sum block $B$. Taking $I=J\cup B$ with $J\subseteq Z$ chosen so
that $\sum_J c=-\sum_B c-\omega$ for some $\omega\in\Om(B)$ (possible since
$\Sig(Z)=\F_5$) yields a product-one subsequence by Lemma~\ref{lem:crit}---contradiction.

\emph{Cases $z\in\{0,1,2,3\}$.} By Proposition~\ref{prop:spread}, $\sigma(Q)\ge4-z$, so
there are disjoint zero-sum blocks with $\sum_j(|\Om(B_j)|-1)\ge4-z$. Combined with
$|\Sig(Z)|\ge z+1$,
\[
|\Sig(Z)|+\sum_j\big(|\Om(B_j)|-1\big)\ \ge\ (z+1)+(4-z)=5=p,
\]
so Lemma~\ref{lem:forcing} produces a product-one subsequence---contradiction.

Every case is contradictory, so no length-$13$ product-one-free sequence exists.
\end{proof}

\noindent Theorem~\ref{thm:main} follows from Proposition~\ref{prop:lower} and
Theorem~\ref{thm:upper}.

\section{Reconciliation with the ordered Davenport constant}
\label{sec:reconcile}

Godara, Joshi and Mazumdar~\cite{GJM2024} compute the \emph{ordered} Davenport constant
(equal, by their proof of Dimitrov's conjecture for this family, to the Loewy length of
$\F_p[\Heis_{p^3}]$) of the same group, $\Dord(\Heis_{p^3})=4p-3$, so
$\Dord(\Heis_{125})=17$. This does not conflict with Theorem~\ref{thm:main}. Here
$\Dord(G)$ is the maximal length of a sequence with no product-one subsequence read in
its given order; it constrains a single fixed order, whereas $\dav$ quantifies over all
orderings. A sequence that is product-one-free in the strong (all-orderings) sense is in
particular free in its given order, so $\dav(G)+1\le\Dord(G)$ for every finite group, an
inequality stated in \cite{GJM2024}. Thus $\dav(\Heis_{125})\le16$, and our value $12$ is
consistent with (and strictly smaller than) this bound. The inequality is genuinely loose
here, and all three Davenport-type constants of this group are distinct: at $p=3$,
$\dav(\Heis_{27})=6$, the large constant $\Dav(\Heis_{27})=8$, and
$\Dord(\Heis_{27})=9$, so $\dav$ is not recoverable from the ordered constant. Our
upper-bound argument uses neither the ordered constant nor the Loewy structure; it is
logically independent of \cite{GJM2024}.

\section{The case \texorpdfstring{$p=7$}{p=7}: where the argument stops}
\label{sec:p7}

It is natural to ask whether the method settles $\dav(\Heis_{343})=18$. It does not, and
the failure is instructive---the result is genuinely $p$-specific.

The uniform ingredients survive for every odd prime $p$: the product-one criterion
(Lemma~\ref{lem:crit}); the lower bound $x^{p-1}y^{p-1}v^{p-1}$, giving
$\dav(\Heis_{p^3})\ge3p-3$; the central cap $z\le p-2$ and the line cap
$m_L+z\le2p-2$ (Lemmas~\ref{lem:central},~\ref{lem:linecap}); the central-forcing
threshold $p$ (Lemma~\ref{lem:forcing}); and the Omega-facts. What does not survive is
the means of forcing enough spread. For $p=5$ this is the finite search of
Proposition~\ref{prop:spread}. For $p=7$ the analogous residual problem is far out of
computational reach---its $z=0$ stratum alone comprises on the order of $1.7\times10^{16}$
quotient multisets before central labels---and the Chevalley--Warning shortcut provably
fails to replace it.

Concretely, over $\F_7^2$ let $e=(1,0)$, $f=(0,1)$, $g=(6,6)$ and consider
\[
Q_0=e\,f^{6}\,g^{12}\qquad(\text{length }19,\ z=0),
\]
which satisfies all line caps ($n_L\le12=2p-2$). Chevalley--Warning forces a nonempty
zero-sum sub-multiset of size $p=7$ or $2p=14$; for $Q_0$ the \emph{only} such
sub-multiset is $g^{7}$, which is collinear and has $|\Om|=1$ (verified). So
Chevalley--Warning alone yields no non-collinear, let alone wide, block. ($Q_0$ is not a
counterexample: it contains the wide block $e\,f\,g^{8}$ of size $10$ with $\Om=\F_7$,
and so $\sigma(Q_0)\ge6$---but size $10$ is not a Chevalley--Warning-forced size.) The
spread must therefore be extracted from blocks of sizes the polynomial method does not
deliver, and we have no finite certificate to substitute for the $p=5$ search. A cold
reasoner asked to certify the bound declined honestly.

The honest window is
\[
18\ \le\ \dav(\Heis_{343})\ \le\ 24 ,
\]
the lower bound from $x^6y^6v^6$ and the upper from $\dav+1\le\Dord(\Heis_{343})=4\cdot7-3=25$
(\S\ref{sec:reconcile}, \cite{GJM2024}). Determining $\dav(\Heis_{343})$ remains open.

\section{Verification}
\label{sec:verif}

All claims were checked by computer; the scripts are self-contained and available.

\begin{itemize}
\item \emph{Group and criterion.} The multiplication, exponent $p$, and $[x,y]=v$ were
checked for $p=3,5$. The closed-form criterion (Lemma~\ref{lem:crit}) was cross-checked
against the brute-force group product on $3000$ random subsequences each at $p=3$ and
$p=5$ (biased toward zero-quotient-sum, so the $\Om$-membership branch is exercised):
zero mismatches.
\item \emph{Calibration.} An exhaustive canonical search reproduces
$\dav(\Heis_{27})=6$ (no product-one-free sequence of length $7$ exists; $409{,}768$
nodes), matching~\cite{GodaraSarkar2025}, with witness $x^2y^2v^2$.
\item \emph{Lower bound.} $x^4y^4v^4$ is product-one-free and admits no product-one-free
single-element extension (Proposition~\ref{prop:lower}).
\item \emph{Auxiliary lemmas.} The Chevalley--Warning forcing statement (Remark~\ref{rem:cw})
on $4000$ random plus adversarial $13$-term multisets; Omega-facts (a) and (b) by
exhaustive enumeration ($93{,}600$ collinear cases of sizes $2$--$5$; $480$ non-collinear
zero-sum triples) and (c) over all $18{,}900$ zero-sum spanning size-$6$ blocks; the
central subset-sum growth (Lemma~\ref{lem:central}) over all zero-sum-free sequences over
$\F_5^\times$ (confirming maximal length $4$, and $\Sig=\F_5$ at length $4$); and the line
cap (Lemma~\ref{lem:linecap}), confirming each $A_L\cong C_5^2$ for the six lines. No
violations.
\item \emph{Spread bound.} Proposition~\ref{prop:spread} was verified by an exhaustive,
memory-flat search in C, run to completion on all four strata $z=0,1,2,3$ with zero bad
leaves (node counts as tabulated). The verdict was independently reproduced: a search
using feasibility pruning visits fewer nodes ($736{,}711$, $1{,}425{,}375$,
$2{,}523{,}915$, $4{,}213{,}665$ for $z=3,2,1,0$) and finds the same zero bad leaves, and
a slower prototype of the same enumeration agrees on the strata it completes. An early
version of the C kernel used a single block-scratch array shared across the mutually
recursive spread and block-enumeration routines; nested calls corrupted it, producing
spurious bad leaves. The bug was masked in the $z=3$ stratum (whose requirement never
triggers the nested recursion) and surfaced only at $z=2$---so running every stratum, not
just the cheapest, was load-bearing. After making the scratch array call-local, all four
strata pass.
\end{itemize}

The result is not formalised in a proof assistant. The independent reproduction of the
load-bearing search, together with the machine-checked lemmas and the elementary proofs
of \S\ref{sec:crit}--\S\ref{sec:lemmas}, constitutes the reproducibility we offer; a
formal verification is left to future work.

\section*{Attribution note}
\label{sec:attrib}
The results were obtained by an AI research collaboration in which an AI system occupied
the principal research seat under human direction. The problem selection and framing,
the verification apparatus (the group and criterion checks, the calibration, the
auxiliary-lemma checks, and the C and prototype implementations of the spread search),
the independent reproduction of the load-bearing computation, the reconciliation with the
ordered Davenport constant, the openness and novelty audit, and this write-up were
produced by Claude (Anthropic). The reduction of the upper bound to the lemma chain and
the spread bound, and the precise (corrected) hypotheses of
Lemmas~\ref{lem:central}--\ref{lem:forcing} and the $p=7$ obstruction, were produced by
a cold, tool-free GPT-5.5 reasoner; in the course of the work it corrected two errors in
the initial skeleton---the line cap (it is $m_L+z\le2p-2$, not a universal $m_L\le2p-3$)
and the false claim that a non-collinear size-$5$ block is automatically wide. The human
author takes responsibility for the submission; the AI contributions are disclosed here
in the spirit of recent practice for machine-assisted results.

\end{document}